\documentclass[reqno,a4paper, draft]{amsart}

\usepackage{amssymb,amscd, amsfonts,  amsthm}

\usepackage[mathscr]{eucal}
\usepackage{graphicx}

\newtheorem{theorem}{Theorem}
\newtheorem{lemma}[theorem]{Lemma}

\newtheorem{corollary}[theorem]{Corollary}
\newtheorem{proposition}[theorem]{Proposition}

\theoremstyle{definition}
\newtheorem{definition}[theorem]{Definition}

\newcommand{\remove}[1]{}

\newcommand{\R}{\mathbb{R}}
\newcommand{\Rn}{\R^n}
\newcommand{\Q}{\mathbb{Q}}
\newcommand{\Qn}{\Q^n}
\newcommand{\Zed}{\mathbb{Z}}
\newcommand{\Zn}{\Zed^n}
\newcommand{\GLn}{\mathsf{GL}(n,\Zed)}
\newcommand{\GLe}{\mathsf{GL}(e,\Zed)}
\newcommand{\SLn}{\mathsf{SL}(n,\Zed)}

\DeclareMathOperator{\aff}{\mathrm{aff}}
\DeclareMathOperator{\conv}{\mathrm{conv}} 
\DeclareMathOperator{\den}{\mathrm{den}}
\DeclareMathOperator{\rank}{\mathrm{rank}}
\DeclareMathOperator{\orb}{\rm orb}

 \title[$\GLn$-orbits]
{Classifying   $\GLn$-orbits of points and rational subspaces}

\author{Leonardo Manuel Cabrer and Daniele Mundici }

\address[L.M. Cabrer]{
Institute of Computer Languages 
Theory and Logic Group\\ 
Technische Universit\"at Wien\\
Favoritenstrasse 9-11\\ A-1040 Vienna \\ Austria}
\email{lmcabrer@yahoo.com.ar}

\address[D. Mundici]{Department of
Mathematics and Computer Science  ``Ulisse Dini'' \\
University of Florence\\
Viale Morgagni 67/a \\
I-50134 Florence \\
Italy}
\email{ mundici@math.unifi.it }

\begin{document}


\begin{abstract}
We first show
 that the subgroup of the abelian real group $\R$ generated by the coordinates of a point in $x\in\Rn$ completely classifies the $\GLn$-orbit of $x$.
This yields a short proof of J.S.Dani's theorem: 
the $\GLn$-orbit of  $x\in\Rn$ 
is dense  iff $x_i/x_j\in \R\setminus \Q$ for some
$i,j=1,\dots,n$.  We then classify $\GLn$-orbits of
rational affine subspaces $F$ of $\R^n$.
We prove that the dimension of $F$ together with the volume of a special parallelotope associated to  $F$
yields  a complete classifier of the $\GLn$-orbit of $F$.
\end{abstract}

\maketitle

\section{Introduction} 
Throughout we   let
$\GLn$ 
denote the group of linear  transformations of the form
$
x\mapsto \mathcal{U}x \,\, \mbox{for} \,\, x\in \Rn,
$
 where   $\mathcal U$ is an integer  $(n \times n)$-matrix 
with
 $\det(\mathcal U)=\pm 1$.   
We let 
$
\orb(x)=\{\gamma(x)\in \Rn\mid
\gamma\in \GLn\}
$
 denote the
{\it $\GLn$-orbit} of $x\in \Rn$.

For all $n=1,2,\dots$ 
we  classify
$\GLn$-orbits of points in  $\Rn$:   
every point in  $\Rn$ is assigned an invariant,
in such a way that two points have the same $\GLn$-orbit iff
they have the same invariant.
For each $x=(x_1,\ldots,x_n)\in \Rn$,
the  invariant   is given by the group
$$
H_x= \Zed x_1+\Zed x_2+\dots+\Zed x_n
$$
generated by $x_1,x_2,\ldots,x_n$ in the
 additive group
$\R$.
 As a first application we give a short 
 self-contained proof of  the characterization
 \cite[Theorem 17]{dan}
 of those  $x\in \Rn$ having a dense  $\GLn$- and
 $\SLn$-orbit. Our proof is also shorter and more elementary
  than the one given in  \cite[Corollary 3.1]{short}.
 
A {\it rational affine hyperplane} $H \subseteq \Rn$
is  a set of the form 
$ 
H =\{z \in \Rn   \mid  \langle h, z\rangle = r\},
$ 
    for some nonzero vector
     $h\in \Qn$ and $r\in \Q$.
     Here $\langle \mbox{-},\mbox{-}\rangle$ denotes scalar product.
When $r=0$ we say that $H$ is a
{\it rational linear hyperplane}.
   A {\it rational affine  (resp., linear) subspace} $A$ of  $\Rn$
is  an  intersection
of rational affine (resp., linear) hyperplanes in $\Rn$.
 
In Theorem~\ref{Theorem:ClassAffineSpace}
we classify  $\GLn$-orbits of rational 
affine subspaces of $\Rn$.

With respect to the vast literature on
orbits of discrete groups acting on $\Rn$
\cite{dan, gui, launog, nog2002, nog2010, wit},
our results highlight the crucial role of (Farey)
regularity of simplicial complexes, i.e., the 
regularity (=nonsingularity=smoothness) of their associated
 fans and toric varieties \cite{ewa,oda}.
Regular simplicial complexes
 were also used in   \cite{cabmun}
to classify  $\GLn\ltimes\Zed^n$-orbits of points.

 \section{Classification of $\GLn$-orbits of points and a
 short proof of J.S.Dani's Theorem}
 \label{section:elementary}
We first show that the  group $H_x $  classifies the
 $\GLn$-orbit of $x\in\Rn$.
 \begin{proposition}
 \label{proposition:complete}
 For all $x,y\in\Rn,\,\,\,$  $H_y=H_x$\,\,\, iff \,\,\,$y=\gamma(x)$
for some $\gamma\in\GLn$. 
 \end{proposition}
 
 \begin{proof}
The $(\Leftarrow)$-direction is trivial. For the converse $(\Rightarrow)$-direction, 
let   
$L_x= \bigcap\{
 L\subseteq\Rn\mid x\in L
 $
 $ \mbox{and $L$ is a rational linear
 subspace of $\Rn$}\}$. 
 Let $e=\dim(L_x)$. Let $v_1,\dots,v_n$ be a basis of the
 free abelian group  $\Zn\subseteq \Rn$, with 
 $v_1,\dots,v_e\in L_x$.  Let  the map $\alpha\in \GLn$ send each
 $v_i$ to the standard $i$th basis vector $\epsilon_i$ of the vector space
 $\Rn$. Let $x'=\alpha(x)$. Then  $x'\in \alpha(L_x)=\{r \in \Rn\mid
r_{e+1}=\dots=r_n=0\}$.  
It follows that $e\geq\rank(H_x)$.  The assumed minimality of
$L_x$ yields $e\leq\rank(H_x),$ whence
 $e=\rank(H_x)$. 
 Trivially,  $H_{x'}=H_x,$ and
$\rank(H_{x'})=e$. The first $e$ coordinates  
$x'_1,\dots,x'_e$  of $x'$ are a basis of  the group $H_{x'}$.

Now, suppose  $y=(y_1,\dots,y_n)\in \Rn$ satisfies $H_y=H_x$, with the intent 
of showing that  $y\in\orb(x)$.  Then $e=\dim(L_x)
=\rank(H_x)=\rank(H_y)=\dim(L_y)$. As above, let
 $w_1,\dots,w_n$ be a basis of $\Zn$, with 
 $w_1,\dots,w_e\in L_y$.  Let  $\beta\in \GLn$ send
 each $w_i $ to $\epsilon_i$. Let $y'=\beta(y)$.  
 The first $e$ coordinates  
$y'_1,\dots,y'_e$  of $y'$ are a basis of  the group $H_{y'}$.
There is  $\delta\in \GLe$ mapping  $x'_j$ to $y'_j$  for each
$j=1,\dots,e$. 
By appending a diagonal of $n-e$ ones to the matrix of $\delta$,
we obtain the matrix of a map  $\hat\delta\in \GLn$ sending
$x'=(x'_1,\dots,x'_e,0\dots,0)$ to
$y'=(y'_1,\dots,y'_e,0\dots,0)$. The composite map
$\gamma=\beta^{-1}\hat\delta\alpha\in\GLn$ satisfies  $\gamma(x)=y$.
 \end{proof}
 \begin{proposition}
 \label{proposition:dense} For any $x\in\Rn$ let
  $e=\rank(H_x)$ and
 $\,\mathcal B_x$ be the set of all (ordered) bases
 $  b=(b_1,\dots,b_e)$  of $H_x$.  
 Suppose  $\emptyset = \R x\cap \Zed^n$,
 where  $ \R x = \{\rho x\mid \rho \in \R\} $.
  Then:
 \begin{itemize}
 \item[(a)] The origin  $0\in\R^e$ is an accumulation point of the set  $\mathcal B_x$,
 identified  with
 a subset of $\R^e$.
  \item[(b)] More generally, for each  $i=1,\dots,e$,  every
   point on the $i$th axis  $\R \epsilon_i$ of  $\R^e$
  is an accumulation point of $\mathcal B_x$.
   \item[(c)] $\mathcal B_x$ is dense in $\R^e$.
    \item[(d)]  Both $\orb(x)$ and the $\SLn$-orbit of $x$ are dense in $\Rn$. 
    \end{itemize}
 \end{proposition}

 \begin{proof}  By assumption $n\geq\rank(H_x)\geq 2$.
 Further, for no
 integer vector $p\in\Zn$ the point  $x$ belongs to
  the set  $\R_{\geq 0}p=
 \{\rho p\mid0\leq \rho\in \R\}$.  

(a) Pick
 $  g=(g_1,\dots,g_e)\in\mathcal B_x$. Fix  $i\in \{1,\dots,e\}$.
 Our assumption about $x$ yields an index 
 $j=j_i$ with  $g_i/g_j\notin \Q$.
Then the
  euclidean algoritm yields a
  sequence  $g_{it}, g_{jt}$ of real numbers such that  $\lim_{t\to\infty} g_{it}
  =0=\lim_{t\to\infty} g_{jt},$   and for each  $t$ the elements  $g_{it},g_{jt}$
  together with the remaining $g_k  \,\,\,(k\not= i,j)$ of $  g$ form
  a basis of  $H_x$. Proceeding inductively on $i$ 
  we obtain a sequence of ordered bases  $  g_t$ of $H_x$ such that
  $\lim_{t\to \infty}   g_t=(0,0,\dots,0)\in \R^e$.
  
  (b) It suffices to argue for $i=1$.  There is  $j=2,\dots,e$  such that  $g_1/g_j\notin
  \Q$, say  $j=2$ without loss of generality. Part (a) yields a sequence of bases
  $  g_t=(g_{1t},\dots,g_{et})$ converging to the origin of $\R^e$.
  Let $z$ be an arbitrary point on the first axis  $\R \epsilon_1$  of $\R^e$.
  For some  $\xi\in \R$ we may write  $z=\xi \epsilon_1$.  Let  $B_\epsilon(z)
  \subseteq \R^e$ be
  the open ball of radius  $\epsilon>0$ cantered at $z$. Then for all suitably
  large  $t$ there is  an integer  $k_t$  such that the point
  $(k_tg_{2t}+g_{1t}, g_{2t},\dots, g_{et})$ belongs to $B_\epsilon(z)$. 
  
  (c)  Fix  $u \in \R^e$  and  $\epsilon>0$.
   Again let  $B_\epsilon(u)
  \subseteq \R^e$ be
  the open ball of radius  $\epsilon>0$ centered at $u$.  In view of (a)
it suffices to argue in case  $u\not=0$.
Then there is a vector  $q\in \Zed^e$
and a point $z $  lying in  $\R_{>0}q\cap B_\epsilon(u)$.
(The intersection of the sphere S$^{e-1}\subseteq \R^e$ 
with the set of rational halflines in $\R^e$ with vertex at the origin, is a dense subset of S$^{e-1}$). 
We may insist that $q$ is {\it primitive} i.e., the gcd of its coordinates is 1.  Since  $q$ can be extended to a basis of the free abelian
group $\Zn$ then some    $\delta\in \GLe$
maps  $q$ to  the first basis vector 
  $\epsilon_1$ of the vector space $\R^e$.
For some  $0 < \omega\in \R$ we can write  $\delta(z)=\omega \epsilon_1$.
As proved in  (b),  $\omega \epsilon_1$ is an accumulation point of the set $\mathcal B_x$ of
bases $  b=(b_1,\dots,b_e)$ of $H_x$. Since $\delta^{-1},$ as well as $\delta,$
is continuous and one-one,  $z$ is an accumulation point of the 
set $\Delta$ of $\delta^{-1}$-images
of these bases. Thus some basis  $(h_1,\dots,h_e)$  of  $H_x$ 
belongs to  $B_\epsilon(u)$. 

(d)  
Let $h=(h_1,\dots,h_e,0,\dots,0)\in\Rn$. Then  $H_h=H_x$.
By Proposition~\ref{proposition:complete},  $h\in \orb(x)$.
Since $e\geq 2$ (by possibly exchanging $h_1$ and $h_2$) we may insist that there exists $\gamma\in\SLn$ such that $\gamma(x)=h$. 
Identifying  $\R^e$  with the subspace
of  $\Rn$ of points whose last $n-e$ coordinates are zero,
we see that $\orb(x)$ is dense in  $\R^e$.
Let $L$ be an arbitrary
rational $e$-dimensional
 linear subspace  $L$ of  $\Rn$. 
Since $e\geq 2$, \ 
  $L$ is the  $\eta$-image of $\R^e$ for some
 $\eta\in \SLn\subseteq\GLn$.  Thus  
 $\orb(x)$ is dense in  $L$ as well. 
We conclude that  $\orb(x)$ and the $\SLn$-orbit of $x$ are dense in $\Rn$. 
 \end{proof}
 
 \begin{corollary}
{\rm \cite[Theorem 17]{dan}}
\label{corollary:cortissima}
The  $\GLn$-orbit (equivalently, the $\SLn$-orbit)
of $x\in \Rn$ is dense in $\Rn$\,  iff\,\,
$\R x \cap \Zed^n=\emptyset$
\,\,iff\,\, for no\,  $p\in\Zn$, $x\in \R_{\geq 0}p$. 

\end{corollary}

\begin{proof} If $x=0$ we have nothing to prove.
Suppose  $x\not=0$ belongs to some rational halfline  $\R_{\geq 0}p$, say
$x=\xi p$ for some  $0\leq \xi$ and  $p\in \Zn$.
 We may assume $p$ primitive.
  Then every  $\gamma\in\GLn$ will send
$p$ to some primitive integer vector $\gamma(p)$ of  $\Rn,$ and
$\gamma(x)$ will have the form  $\xi\gamma(p)$. Since the length
of the shortest primitive integer vector in $\Rn$ is 1, it follows that
$\orb(x)$ is disjoint from the open ball $B_\xi(0)$. The same argument
holds a fortiori for $\SLn$-orbits.
The converse follows from Proposition~\ref{proposition:dense}(d).
\end{proof}

\section{Regular simplexes and the $\lambda_i$-measure of a rational polyhedron}
\label{section:preliminaries}

In this section we present the necessary material for our classification of $\GLn$-orbits of rational affine spaces. We refer the reader to \cite{ewa, sta} for elementary background on simplicial complexes and polyhedral topology.

For any  set $E\subseteq \Rn$
  the   {\it affine hull}  
 $\,\,\aff(E)\,\,$    is the set of all {\it affine combinations} in  $\Rn$
of elements of $E$.  

A finite set  $\{z_1,\ldots,z_m\}$ of  points in
$\Rn$ is said to be
{\it affinely independent}  if    none of its  elements 
 is  an affine combination of the remaining elements.
 One then easily sees that a subset $R$ of  $\Rn$
is  an $m$-dimension rational affine subspace
 of $\Rn$ iff there exist 
 affinely independent
 $v_0,\ldots,v_m\in\Qn$ such that 
 $R=\aff(v_0,\ldots,v_m)$.
 For $\,\,0\leq m\leq n$, an {\it m-simplex}
in $\Rn$ is the
{\it convex hull} $T = \conv(v_{0},\ldots,v_{m})$ of $m+1$ affinely
independent points $v_{0},\ldots,v_{m}\in \Rn$.  The {\it  vertices}
$v_{0},\ldots,v_{m}$ are uniquely determined by $T$.
$T$ is said to be a {\it rational simplex}  if its vertices are
rational.  
A {\it rational polyhedron}  $P$  in $\Rn$ is the union
of finitely many rational simplexes  $T_i$  in $\Rn$.
$P$ need not be convex or connected.
The $T_i$ need not have the same dimension.

By  the  {\it denominator}  $\den(x)$ of a rational point
  $x=(x_1,\dots,x_n)$
 we understand the least common denominator 
  of its coordinates. 
The  {\it homogeneous correspondent}
of a rational point $x=(x_1,\ldots,x_n)\in \Qn$ is the 
integer  vector 
\[
 \widetilde{ x } = \bigl(\den(x)\cdot x_1,
 \ldots, \,\den(x)\cdot x_n,\,\,\den(x)\bigr) 
 \in \Zed^{n+1}.
\]

Let  $n =1,2,\ldots$  and  $m=0,\ldots,n$.
Following \cite{mun-cpc},
a rational $m$-simplex  
$T=\conv(v_0,\ldots,v_m)\subseteq \Rn$
is said to be  {\it {\rm(}Farey{\rm)} regular} 
if the set $\{\tilde v_0 ,\ldots, \tilde v_m\}$ of 
homogeneous correspondents of the vertices of $T$
can be extended to a   basis  of
the free abelian group $\Zed^{n+1}$.

By a  {\it simplicial complex} in $\Rn$ we mean a finite set 
$\Delta$ of 
simplexes $S_i$ in $\Rn$, closed 
under taking faces,
 and having the further property that any two 
elements of $\Delta$ intersect in a common face.  
The complex $\Delta$ is said to be {\it rational} if the vertices
of all $S_i\in \Delta$ are rational.
For every complex $\Delta$, its {\it support} $|\Delta|
\subseteq \Rn$ is the
pointset union of all 
simplexes of $\Delta$.   
In this case we say that $\Delta$ is a triangulation of $|\Delta|$.
We  say that 
$\Delta$ is  {\it regular} if 
 every simplex of $\Delta$ is 
regular.

\subsection{The rational measure  $\lambda_d\,$}
\label{subsection:rational-measure}
For fixed $n=1,2,\dots$   let
  $P \subseteq\Rn$ be a  
  (not necessarily rational) polyhedron.
For any  
triangulation $\Delta$ of $P$
and $i=0,1,\ldots$  let 
  $\Delta^{\max}(i)$ denote   the set of
maximal $i$-simplexes of $\Delta$.
The
{\it $i$-dimensional part $P^{(i)}$ of $P$}
is now defined by
\begin{equation}
\label{equation:part}
P^{(i)}=\bigcup\{T\mid T\in \Delta^{\max}(i)\}.
\end{equation}
%
%
%
%
Since 
   any two triangulations of $P$ have a
  joint subdivision, the definition of 
$P^{(i)}$ does not depend on the chosen
  triangulation $\Delta$  of $P$.
If   $P^{(i)}$ is nonempty, then it is  
 an $i$-dimensional polyhedron
 (i.e., a finite union of $i$-simplexes in $\Rn$)
 whose   $j$-dimensional part  $P^{(j)}$ 
 is empty for each   $j\not= i$.
 Trivially, $P^{(k)}=\emptyset$  for each
 integer $k>\dim(P)$.

For every 
 regular $m$-simplex  
$S=\conv(x_0,\ldots,x_m)\subseteq  \Rn$
we use the notation  
\[
\den(S)=\prod_{j=0}^m  \den(x_j).  
\]
Let $R$ be a  {\it rational} polyhedron   $\Rn$.
In \cite[Lemma 2.1]{mun-cpc} it is proved that $R$ has a regular triangulation.
For every 
regular triangulation $\Delta$ of $R$,
and $i=0,1,\ldots,$  the rational number
  $\lambda(n,i,R,\Delta)$  is defined by
\[
\lambda(n,i,R,\Delta)=
\begin{cases}
 \sum_{T\in \Delta^{\max}(i)}
\frac{1}{i!\,\,\den(T)} & \mbox{ if } \Delta^{\max}(i)\not=\emptyset\\[0.3cm]
0 &   \mbox{ if }  \Delta^{\max}(i)=\emptyset.
\end{cases}
\]%
As proved in  \cite[Theorem 2.3]{mun-cpc}, 
$\lambda(n,i,R,\Delta)$
does not depend on the regular triangulation   $\Delta$ of $R$.
The proof of this result
 relies upon the 
 solution of the weak Oda conjecture by  
Morelli and W{\l}odarczyk 
\cite{mor, wlo}.
Thus, for any rational polyhedron
$R\subseteq \Rn$ and $i=0,1,2\dots,$  we can unambiguously write
\[
\lambda_i(R) = \lambda(n,i,R,\Delta),
\]
where  $\Delta$  is an arbitrary regular
triangulation of $R$. 
We say that $\lambda_i(R)$  is the
{\it {$i$-dimensional}   rational measure}
of $R$.  Trivially,  $\lambda_i(R)=0$
for each integer  $i>\dim(R)$.

\bigskip
For a characterization of $\lambda_i$
let  
\[
\mathcal G_n=\GLn \ltimes \Zn
\]
denote  the
   group of transformations
of the form
$
x\mapsto \mathcal{U}x + b   \quad (x\in \Rn),
$
 where $b\in \Zn$
 and  $\mathcal{U}$ is an  integer $(n \times n)$-matrix with   determinant
 $\pm 1$.  
 $\mathcal G_n$ is known as the $n$-dimensional affine group over the
 integers.  Also let 
 $
\mathscr P^{(n)}
$
denote the set of all 
  rational polyhedra in $\Rn$.

Building on the main
result of
\cite{mun-dcds}, in
 \cite[1.1, 4.1, 6.2]{mun-cpc}  the following result is proved:

\begin{theorem}
\label{theorem:volumi}
For  each $n=1,2,\ldots$ and 
$\,\,i=0,1,\ldots, $  the map
$\lambda_i\colon \mathscr P^{(n)}\to \R_{\geq 0}$
has the following properties,
for all $P,Q \in \mathscr P^{(n)}:$

\begin{description}
\item[(i) \  $\mathsf{Invariance}$]
  If
  $P=\gamma(Q)$ for some  $\gamma\in \mathcal G_n$
then $\lambda_{i}(P)=\lambda_{i}(Q)$. 

 \smallskip

\item[(ii) $\mathsf{Valuation}$] 
$\lambda_{i}(\emptyset)=0,\,\,$
$\lambda_{i}(P)=\lambda_i(P^{(i)})$,  and the restriction of 
 $\lambda_{i}$ to the set of all 
 rational polyhedra  $P,Q$  in $ \Rn$ 
 having dimension  $\leq i$ is a {\rm valuation}: in other words,
\[
		\lambda_{i}(P)+\lambda_{i}(Q)=
			\lambda_{i}(P\cup Q)+
			\lambda_{i}(P\cap Q).
\]
 
  \smallskip

\item[(iii) $\mathsf{Conservativity}$]  
Let  $(P,0)=\{(x,0)\in \R^{n+1}\mid x\in P \}$.
Then $\lambda_i(P)=\lambda_i(P,0)$. 

  \smallskip

\item[(iv) $\mathsf{Pyramid}$]  
For
$k=1,\ldots,n$, if
  $\,\,\conv(v_0,\ldots,v_k)$
is a 
regular $k$-simplex in $\Rn$
with  $v_0\in \Zn$ then
\[
\lambda_k(\conv(v_0,\ldots,v_k))=
 {\lambda_{k-1}(\conv(v_1,\ldots,v_k))}/{k}.
\]
 
   \smallskip
   
\item[(v) $\mathsf{Normalization}$]  
Let 
$j=1,\ldots,n$.  Suppose  the set 
  $B=\{w_1,\ldots,w_j\} \subseteq  \Zn$
is part of a basis of the free abelian 
  group $\Zn$.
  Let the closed   
parallelotope $\mathcal P_B\subseteq \Rn$
  be defined by
\[
\textstyle \mathcal P_B =\left\{x\in \Rn\mid x=
  \sum_{i=1}^j \gamma_iw_i,\,\,\, 0\leq \gamma_i\leq 1\right \}.
\]
  Then  $\lambda_j(\mathcal P_B)=1$.
 
 \smallskip
   
  \item[(vi) $\mathsf{Proportionality}$]
  Let 
$A$ be an $m$-dimensional rational affine subspace of $\Rn$  for some $m=0,\ldots,n$.  
 Then there is a constant $\kappa_A>0$,
 only depending on $A$,   such that 
  $\,\,\,\lambda_m(Q) = \kappa_A\cdot \mathcal H^m(Q)$
  for every rational 
  $m$-simplex
   $Q\subseteq A$.
Here as usual,  $\mathcal H^m$ denotes the $m$-dimensional
 Hausdorff measure, \cite{fed}. Moreover, if $m=n$, then $\kappa_A=1$
\end{description}

Conversely, conditions (i)-(vi)  uniquely characterize the maps $\lambda_i$
among all maps $\mu \colon \mathscr P^{(n)}\to \R_{\geq 0}$. 
\end{theorem}

\section{Orbits of rational affine subspaces of $\Rn$: Preliminaries}
\label{section:preliminaries-leo}

\begin{lemma}
\label{Lemma:GammaFromSimplex}
For each rational point 
$x\in\Q^n$
and  $\psi\in\GLn$, 
$\den(x)=\den(\psi(x))$.
      Moreover, if
$\conv(0,v_{1},\ldots,v_{n})\subseteq \R^{n}$ and
$\conv(0,w_{1},\ldots,w_{n})\subseteq \R^{n}$ 
are  
 regular
$n$-simplexes, 
the following conditions are equivalent:
\begin{itemize}
\item[(i)] $\den(v_{i})=\den(w_{i})$ for all
$i\in\{1,\ldots,n\}$;

\smallskip
\item[(ii)] there exists $\gamma\in\GLn$ such that $\gamma(v_i)=w_i$ for each $i\in\{1,\ldots,n\}$.
\end{itemize}
\end{lemma}

\begin{proof}
(ii)$\Rightarrow$(i) is trivial.
For the converse direction (i)$\Rightarrow$(ii), from
the  regularity
of 
$\conv(0,v_{1},\ldots,v_{n})$ and
$\conv(0,w_{1},\ldots,w_{n})$ 
we obtain
    bases $\{\tilde{0},\tilde{v}_{1},\ldots,
\tilde{v}_{n}\}$ 
and 
$\{\tilde{0},\tilde{w}_{1},\ldots,
\tilde{w}_{n}\}$
  of the free
abelian group $\Zn$. (Note $\tilde{0}=(0,\ldots,0,1)$).
There are   integer ($(n+1)\times(n+1)$)-matrices
$\mathcal C$ and $\mathcal D$
such that $\mathcal C \tilde{v}_i=\tilde{w}_i$
 and 
$\mathcal D \tilde{w_i}=\tilde{v}_i$ for each $i\in\{0,1,\ldots,n\}$.
It follows that $\mathcal C \cdot \mathcal D$ and 
$\mathcal D\cdot \mathcal C$ are equal to the identity matrix. 
For each 
$i\in\{0,1,\ldots,n\}$ the
 last coordinate of $\tilde{v}_i$, as well as the last coordinate
 of $\tilde{w}_i$,  coincide
 with  $\den(v_i)=\den(w_i)$. Thus $\mathcal C$ and $\mathcal D$ 
  have the form
\[
\mathcal{C}=\left(\begin{tabular}{c|c} $\mathcal U$ & $0$ \\ \hline $0,\ldots,0 $ &
$1$ \end{tabular}\right) \quad \mathcal{D}=\left(\begin{tabular}{c|c} $\mathcal V$ & $0$ \\ \hline $0,\ldots,0 $ &
$1$ \end{tabular}\right) 
\] 
 for some  integer $(n\times n)$-matrices $\mathcal U$ and $\mathcal V$.
For each  $i=0,1,\ldots,n$ we have 
$\mathcal{C}\tilde{v}_{i}= \mathcal C(\den(v_{i})\cdot(v_{i},1))
=\den(w_{i})\cdot(\mathcal Uv_{i},1).  $ 
From
$\mathcal C\tilde{v}_{i}=\tilde{w}_{i}=\den(w_{i})\cdot(w_{i},1)$ 
we obtain
$\mathcal Uv_{i}=w_{i}$.
Similarly,  $\mathcal Vw_i =v_i$. 
In conclusion,
both maps $\gamma(x)=\mathcal Ux$ 
and $\mu(y)=\mathcal Vy$ are in $\GLn$.
Further, 
$\gamma^{-1}=\mu$ and $\gamma(v_i)=w_i$ for each $i\in\{0,1,\ldots,n\}$.
\end{proof}

 \medskip

\subsubsection*{Notation.}
For each $w\in \Q^n$ we use the notation 
\[
||w||_{\Q}=\lambda_1(\conv(0,w)),
\]
 where
$\lambda_1$ is the 1-dimensional rational measure 
introduced  in  Subsection~\ref{subsection:rational-measure}.
Further, for
  every  nonempty rational affine space $F\subseteq \R^n$
 we let 
\[
d_F=\min\{\den(v)\mid v\in  F\cap \Qn\}.
\]
 
 \medskip

\begin{lemma}
\label{lemma:wlog-isobunch}  
For  fixed $n=1, 2,3,\ldots$ and $e=0,\ldots,n,$
let  $F$ be  an $e$-dimensional rational affine space in 
 $\Rn$.
Let  
 $v_0\in F\cap \Qn$  with  $\den(v_0)=d_F$.
Then there
are rational points
 $v_1,\ldots,v_e\in F$, all with denominator $d_F$,
 such that $\conv(v_0,\ldots,v_e)$ is a 
regular
 $e$-simplex. 
\end{lemma}

\begin{proof}
 Starting from any rational
$e$-simplex $R\subseteq F$ 
having $v_0$ among its vertices and using
\cite[Lemma 2.1]{mun-cpc}, one immediately obtains
 rational points
$w_1,\ldots,w_e\in F$ such that $\conv(v_0,w_1,\ldots,w_e)$ is regular. 
If 
$\den(w_i)>d_F$
for some $i\in\{1,\ldots,e\}$,  then
let   the integer $m_i$  be uniquely determined by writing
  $m_i \cdot d_F<\den(w_i)\leq (m_i+1) \cdot  d_F$.
 Let $v_i\in F\cap\Qn$ be the unique rational point 
 with  $\tilde v_i=\tilde w_i-m_i \tilde v_0$. Then 
$d_F\leq\den(v_i)=\den( w_i)- m_i \cdot 
 \den(v_0)=\den( w_i)- m_i  \cdot  d_F\leq d_F$.
The  new $(e+1)$-tuple of integer vectors 
$(\tilde v_0,\tilde v_1,\ldots,\tilde v_e)$ can be completed to a basis of $\Zed^{n+1}$,
precisely  as 
 $(\tilde v_0,\tilde w_1\ldots,\tilde w_e)$ does.
Thus  $\conv(v_0,\ldots,v_e)$ is regular.
\end{proof}

\section{Classification of rational affine subspaces of $\Rn$: Conclusion}

Generalizing the above definition of $\mathcal P_B$ (Theorem~\ref{theorem:volumi}(v)),
given
  $v_1,\ldots,v_k\in \R^n$  with $k\leq n$ we let $\mathcal{P}(v_1,\ldots,v_k)$ denote the closed 
$k$-parallelotope generated by $v_1,\ldots,v_k$, in symbols,
\[\textstyle\mathcal{P}(v_1,\ldots,v_k)=\left\{\, \sum_{i=1}^{k}\alpha_i v_i\in \R^n\mid \alpha_1,\ldots,\alpha_k\in[0,1]\,\right\}.\]

\begin{lemma}\label{Lem:VF}
Let  $F\subseteq \R^n$ be an 
$e$-dimensional rational affine space
($e=0,\dots,n$). Let the points
 ${v_0,\ldots,v_e,w_0,\ldots,w_e\in F}$
 satisfy  $\den(v_i)=\den(w_i)=d_F$ for each $i=0,\ldots,e$. Suppose further that both
  $\conv(v_0,\ldots,v_e)$ and $\conv(w_0,\ldots,w_e)$ are 
 regular  $e$-simplexes. 
Then 
$\lambda_{e+1}(\mathcal{P}(v_0,\ldots,v_e))=\lambda_{e+1}(\mathcal{P}(w_0,\ldots,w_e))$. 
\end{lemma}

\begin{proof}
Recall the notation of \eqref{equation:part} for the
 $i$-dimensional
part of a polyhedron.
If $0\in F$, then $(\mathcal{P}(0,v_0,\ldots,v_e))^{(e+1)}=\emptyset=(\mathcal{P}(0,w_0,\ldots,w_e))^{(e+1)}$,
and hence
\[\lambda_{e+1}(\mathcal{P}(v_0,\ldots,v_e))=0=\lambda_{e+1}(\mathcal{P}(w_0,\ldots,w_e)).
\]
Now assume that $0\notin F$.
Let $G$ be the linear subspace of $\R^n$ generated by $F\cup\{0\}$. 
Then $G$ is an $(e+1)$-dimensional rational subspace and $d_G=1$. 
By Lemma~\ref{lemma:wlog-isobunch}  there exist 
$z_1,\ldots, z_{e+1}\in G$ 
such that $\den(z_i)=1$ 
for each $i=1,\ldots,(e+1)$ and
 $\conv(0,z_1,\ldots, z_{e+1})$ is a regular $(e+1)$-simplex.
Let $\varepsilon_1,\ldots,\varepsilon_n$ be the canonical basis of 
the vector space $\R^n$. 
  Lemma~\ref{Lemma:GammaFromSimplex} yields 
$\gamma\in\GLn$ such that $\gamma(z_i)=\varepsilon_i$ 
for each $i=1,\ldots,(e+1)$. 
Then $\gamma(F)
=F'=\R \varepsilon_1+\cdots+\R\varepsilon_{e+1}$,
whence  there exist $v'_0,\ldots,v'_e,w'_0,\ldots,w'_{e}\in \R^{e}$
such that
$ \gamma(v_i)=(v'_i,0,\ldots,0)$ and $ \gamma(w_i)=(w'_i,0,\ldots,0)$ for each $i=0,\ldots,e$.
By  Theorem~\ref{theorem:volumi} (i) and (iii),
\begin{equation}\label{Eq:reduce}
\lambda_{e+1}(\mathcal{P}(v_0,\ldots,v_e))=\lambda_{e+1}(\gamma(\mathcal{P}(v_0,\ldots,v_e)))=\lambda_{e+1}(\mathcal{P}(v'_0,\ldots,v'_e))).
\end{equation}
Similarly, 
\begin{equation}
\lambda_{e+1}(\mathcal{P}(w_0,\ldots,w_e))=\lambda_{e+1}(\mathcal{P}(w'_0,\ldots,w'_e))).
\end{equation}
From  $\mathcal{P}(v'_0,\ldots,v'_e),\mathcal{P}(w'_0,\ldots,w'_e)\subseteq\R^e$ we get 
\begin{equation}
\lambda_{e+1}(\mathcal{P}(v'_0,\ldots,v'_e))
=\mathcal H^{e+1}(\mathcal{P}(v'_0,\ldots,v'_e))
\end{equation}
and
\begin{equation}
\lambda_{e+1}(\mathcal{P}(w'_0,\ldots,w'_e))
=\mathcal H^{e+1}(\mathcal{P}(w'_0,\ldots,w'_e))
\end{equation}
Theorem~\ref{theorem:volumi}(vi) yields a real 
  $\kappa_{F'}$ such that
\[
\mathcal H^e(\conv(v'_0,\ldots,v'_e))=\kappa_{F'} \lambda_{e}(\conv(v'_0,\ldots,v'_e))=\kappa_{F'} \frac{1}{e! d_F^{e+1}}.
\]
Similarly,  
$\mathcal H^e(\conv(w'_0,\ldots,w'_e))
=\kappa_{F'} \frac{1}{e! d_F^{e+1}}$.
 Then 
 \[\mathcal H^e(\conv(v'_0,\ldots,v'_e))
=\mathcal H^e(\conv(w'_0,\ldots,w'_e)).
\] 
Since $v'_i,w'_i\in F'$ for each $i=0,\ldots,e$, we can write
\begin{equation}\label{Eq:volume}
\mathcal H^{e+1}(\conv(0,v'_0,\ldots,v'_e))=\mathcal H^{e+1}(\conv(0,w'_0,\ldots,w'_e)).
\end{equation}
Finally, combining \eqref{Eq:reduce}-\eqref{Eq:volume}, we obtain
\begin{align*}
\lambda_{e+1}(\mathcal{P}(v_0,\ldots,v_e))&
=\lambda_{e+1}(\mathcal{P}(v'_0,\ldots,v'_e)))
= \mathcal H^{e+1}(\mathcal{P}(v'_0,\ldots,v'_e)) 
 \\
&= \mathcal H^{e+1}(\mathcal{P}(0,w'_0,\ldots,w'_e))
=\lambda_{e+1}(\mathcal{P}(w'_0,\ldots,w'_e)))\\
&=\lambda_{e+1}(\mathcal{P}(w_0,\ldots,w_e)).\qedhere
\end{align*}
\end{proof}

\medskip
\begin{definition}\label{Def:VF}
For every  
$e=0,1,\ldots,n$ and
$e$-dimensional rational affine space $F\subseteq \R^n$
 we let 
\[
  \mathsf V_F=
\lambda_{e+1}\bigr(\mathcal{P}(v_0,\ldots,v_e)\bigl),
\]
where  $v_0,\dots,v_e$ are rational points of $F$ of denominator
$d_F$ such that the $e$-simplex  $\conv(v_0,\dots,v_e)$ is
 regular.
The existence of these points is guaranteed by 
Lemma~\ref{lemma:wlog-isobunch}, 
and the independence of $\mathsf V_F$ 
from the choice of $v_0,\ldots,v_e$ follows from  Lemma~\ref{Lem:VF}.
\end{definition}

As usual, given  $\kappa \in\R$ and $S\subseteq \Rn$ we let
$\kappa S=\{\kappa x\mid x\in S\}$.

\begin{lemma}\label{Lem:VolDFmultiple}
For every   $e$-dimensional rational affine space $F\subseteq \R^n$
($e=0,\dots,n$),
letting  $d=d_F$ we have
$ \mathsf V_{d F}=d^{e+1}  \mathsf V_F$.
\end{lemma}
 \begin{proof}
Observe that $\mathsf V_{F}=0$ iff $0\in F$. If this is the case then the result follows trivially since $0\in d F$.
So assume $0\notin F$, whence  $e<n$.
In view of  Lemma~\ref{lemma:wlog-isobunch},  
let  $v_0,\dots,v_e$ be rational points of $F$ of denominator
$d$ such that the $e$-simplex  $\conv(v_0,\dots,v_e)$ is regular.
Let $G= d F$. Then
  $d v_0,\ldots, d v_e \in\Zn\cap G$.

We claim that $\conv(d v_0,\dots,d v_e)$ is regular.
For otherwise, there exists 
$v\in \conv(d v_0,\dots,d v_e)\cap \Q^n$ such that 
$\den(v)\leq e$. 
It follows that
 $w=\frac{1}{d}v\in F$,\ $w\in \conv(v_0,\dots,v_e)$ 
and $\den(w)\leq d \den(v)\leq d e$, 
thus contradicting the regularity of  $\conv(v_0,\dots,v_e)$.

Let 
$A$ be the  $(e+1)$-dimensional 
 rational linear subspace of $\Rn$  generated by $v_0,\ldots,v_e$.  
  Theorem~\ref{theorem:volumi}
 yields a constant $\kappa_A>0$,
 such that 
  $\,\,\,\lambda_{e+1}(Q) = \kappa_A\cdot \mathcal H^{e+1}(Q)$
  for every rational 
  $(e+1)$-simplex
   $Q\subseteq A$. Since both $\mathcal{P}(v_0,\ldots,v_e)$ and $\mathcal{P}(dv_0,\ldots,dv_e)$ are finite unions of $(e+1)$-simplexes 
  we conclude
 \begin{align*}
\mathsf{V}_{G}&=\lambda_{e+1}(\mathcal{P}(dv_0,\ldots,dv_e))=
\kappa_A\cdot \mathcal H^{e+1}(d\mathcal{P}(v_0,\ldots,v_e))\\
&=\kappa_Ad^{e+1}\cdot \mathcal H^{e+1}(\mathcal{P}(v_0,\ldots,v_e))=d^{e+1}\lambda_{e+1}(\mathcal{P}(v_0,\ldots,v_e))=d^{e+1}
\mathsf V_{F}. 
\qedhere
\end{align*}
\end{proof}

\begin{lemma}
\label{Lem:Filetto}
Every
$e$-dimensional rational affine space
$F\subseteq \R^n$\,\, {\rm (}$e=0,\dots,n${\rm )},
  contains a point $v\in F\cap\Qn$ 
such that 
$\,\,\den(v)=d_F\,\mbox{ and }\,
 d_F^e \mathsf V_F=||v||_{\Q}$.
\end{lemma}

\begin{proof}

If $ \mathsf V_F=0$ then $0\in F$ and $v=0$  will do.

If $\mathsf V_F\neq 0$
let us first assume  $d_F=1$.
By the same argument used in 
Lemma~\ref{Lem:VF}
 we can assume
 $\dim(F)=n-1$ without loss of generality.

Let $v_0,\ldots,v_{n-1}\in F\cap\Zn$ 
be such that $\conv(v_0,\ldots,v_{n-1})$ is  a   regular  $e$-simplex.
Assume that $v_0,\ldots,v_{n-1}$ are ordered  in such a way that
the determinant of the matrix $\mathcal U$ whose rows are 
$v_0,\ldots,v_{n-1}$ is   $>0$. 
Since $\mathsf V_F = \lambda_{n}(\mathcal{P}(v_0,\ldots,v_n))$, 
by Theorem~\ref{theorem:volumi}(vi) we can write
\[\mathsf V_F=\lambda_{n}(\mathcal{P}(v_0,\ldots,v_{n-1}))=
\mathcal H^{n}(\mathcal{P}(v_0,\ldots,v_{n-1}))=\det(\mathcal U).
\]
Since $\conv(v_0,\ldots,v_{n-1})$ is
regular, then so is $\conv(0,v_1-v_0,\ldots,v_{n-1}-v_0)$.
For each $i=1,\ldots, n-1$, ${v_i-v_0}\in\Zn$, 
and hence $\den(v_i-v_0)=1$. 
Therefore, there exists $(z_1,\ldots,z_{n+1})\in \Zed^{n+1}$ such that 
\[\{(0,\ldots,0,1),(z_1,\ldots,z_{n+1}),(v_1-v_0,1),\ldots, (v_{n-1}-v_0,1)\}\] is a basis of $\Zed^{n+1}$. Without loss of generality we
may assume that  $z_{n+1}=1=$ the value of the determinant of the matrix $\mathcal W$ 
whose rows are 
\[(0,\ldots,0,1),(z_1,\ldots,z_{n+1}),(v_1-v_0,1),
\ldots, (v_{n-1}-v_0,1).\]
Let $w=(z_1,\ldots,z_n)$. Since $\conv(0,w)$ is regular
then  $||w||_{\Q}=1$.

\medskip 
\noindent {\sf Claim:}  $\mathsf V_F \cdot w \in F$.

As a matter of fact: 
\begin{align*}
\det\left(\begin{tabular}{c}
$(\mathsf V_F \cdot w)-v_0$\\
$v_1-v_0$\\
\vdots\\
$v_{n-1}-v_0$
\end{tabular}\right)&=\mathsf V_F 
\det\left(\begin{tabular}{c}
$w$\\
$v_1-v_0$ \\
\vdots\\
$v_{n-1}-v_0$
\end{tabular}\right)-\det\left(\begin{tabular}{c}
$v_0$\\
$v_1-v_0$\\
\vdots\\
$v_{n-1}-v_0$
\end{tabular}\right)\\[0.4cm]
&=\mathsf V_F 
\det\left(\begin{tabular}{cl}
$0,\ldots,0$& 1\\
$w$& 1\\
$v_1-v_0$& 1 \\
\vdots&\vdots\\
$v_{n-1}-v_0$& 1
\end{tabular}\right)-\det(\mathcal U)\\[0.3cm]
&=\mathsf V_F 
\det(\mathcal W)-\det(\mathcal U) = 0.
\end{align*}
Therefore,
the point  $(\mathsf V_F \cdot w)-v_0$ lies
 in the $(n-1)$-dimensional 
subspace of $\R^n$ generated by 
$(v_1-v_0),\ldots, (v_{n-1}-v_0)$, that is, 
$\mathsf V_F \cdot w\in \aff(v_0,\ldots,v_{n-1}) = F$, 
which settles our claim.

Setting now $v=\mathsf V_F\cdot w$, we have
 $v\in F$ by our claim.  Another application of
  Theorem~\ref{theorem:volumi}(vi)  yields
\[
||v||_{\Q}=\lambda_{1}(\conv(0,\mathsf V_F\cdot w))=\mathsf V_F\lambda_{1}(\conv(0, w))=\mathsf V_F ||w||_{\Q}=\mathsf V_F.
\]

If $d_F\neq 1$, let $G=d_F F$. Observe that $d_G=1$.
Thus there exists $w\in G$ such that $||w||_{\Q}=\mathcal{V}_G$. Let $v\in F$ be such that $w=d_F v$. By  Lemma~\ref{Lem:VolDFmultiple} and  Theorem~\ref{theorem:volumi}(vi),
\[\mathcal ||v||_{\Q}=||w||_{\Q}/d_F=\mathsf V_{G}/d_F=d_F^{e}  \mathsf V_F.\qedhere\]

\end{proof}

We are now in a position to 
classify $\GLn$-orbits of rational
affine subspaces of $\Rn$.
\begin{theorem} 
\label{Theorem:ClassAffineSpace}
Let $F,G\subseteq \R^n$ be nonempty rational affine spaces. Then the
 following conditions are equivalent:
\begin{itemize}
\item[(i)] There exists $\gamma\in\GLn$ such that $\gamma(F)=G$;

\smallskip
\item[(ii)] $(\dim(F), \mathsf V_F)=(\dim(G), \mathsf V_G)$.
\end{itemize} 
\end{theorem}

\begin{proof}
The (i)$\Rightarrow$(ii) implication is trivial.\\
(ii)$\Rightarrow$(i)   
By the same argument used in Lemma~\ref{Lem:VF}
 we can assume without loss of generality that $\dim(F)=\dim(G)=n-1$.
We argue by cases:

\medskip
 
If $ \mathsf V_F=  \mathsf V_G=0$ then $0\in F\cap G$. 
By Lemma~\ref{lemma:wlog-isobunch}, there exist $v_1,\ldots,v_{n-1}\in F$ and  $w_1,\ldots,w_{n-1}\in G$ 
such that $\den(v_i)=\den(w_i)=1$ for each $i=1,\ldots,{n-1}$ 
and  both 
 $(n-1)$-simplexes $\conv(0,v_1,\ldots,v_{n-1})$ and $\conv(0,w_1,\ldots,w_{n-1})$ are regular. 
Therefore, there exist $v,w\in\Zn$ such that 
$\conv(0,v,v_1,\ldots,v_{n-1})$ and $\conv(0,v,w_1,\ldots,w_{n-1})$ 
are regular. 
An application of Lemma~\ref{Lemma:GammaFromSimplex} yields ${\gamma\in\GLn}$ such that $\gamma(v_i)=w_i$ 
for each $i=1,\ldots,n-1$. 
Thus 
\[
\gamma(F)= \gamma(\aff(0,v_1,\ldots,v_{n-1}))
=\aff(\gamma(0),\gamma(v_1),\ldots, \gamma(v_{n-1}))=G.
\]

\medskip

Now assume that $\mathsf V_F=\mathsf V_G\neq 0$.  Let $e=\dim(F)=\dim(G)$. 
By Lemma~\ref{Lem:Filetto}, there exist $v\in F$ and $w\in G$ such that 
\begin{equation}\label{eq:filotti}
||v||_{\Q}= d_F^{n-1}   \mathsf V_F\quad\mbox{ and }\quad||w||_{\Q}=d_G^{n-1}  \mathsf V_G .
\end{equation}

\noindent {\sf Claim 1:} $d_F=d_G$.

By \eqref{eq:filotti}, 
 $d_F^{e+1}\mathsf V_F=||d_F v||=||\den(v) v||_{\Q}\in\Zed$. 
Moreover, if  $k\in\{1,2,\ldots\}$
 is such that 
$k^{e+1}   \mathsf V_F\in\Zed$,
 then $k (\frac{k}{d_F})^{e}  ||v||_{\Q}\in \Zed$. 
Thus $k (\frac{k}{d_F})^{e}  v\in \Zn$, i.e., $k (\frac{k}{d_F})^{e} \in\Zed$ and $d_F$ is a divisor $k (\frac{k}{d_F})^{e}$. 
Therefore, $d_F$ is a divisor of $k$. 
We have proved that 
\[d_F=\min\{k\in\{1,2,\ldots,\}\mid k^{e+1} \mathsf V_F\in\Zed\}.\]
 Similarly  $d_G=\min\{k\in\{1,2,\ldots,\}\mid k^{e+1}   \mathsf V_G\in\Zed\}$. Since by hypothesis
  $\mathsf V_F=\mathsf V_G$, then $d_F=d_G$, with settles
  our claim.

\medskip

By \eqref{eq:filotti} and Claim 1,
$||v||_{\Q}= d_F^{n-1}   \mathsf V_F= d_G^{n-1}  \mathsf V_G =||w||_{\Q}$.  
By Lemma~\ref{lemma:wlog-isobunch}, there exist 
$v_1,\ldots,v_{n-1}\in F$ and  $w_1,\ldots,w_{n-1}\in G$ 
such that $\den(v_i)=\den(w_i)=d_F$ for each $i=1,\ldots,n-1$
 and both $(n-1)$-simplexes $\conv(v,v_1,\ldots,v_{n-1})$
 and $\conv(w,w_1,\ldots,w_{n-1})$ are regular.

Let us set $v'=v/(||v||_{\Q})$ and $w'=w/(||w||_{\Q})$.
By Theorem~\ref{theorem:volumi}(vi),\,\, $||v'||_{\Q}=||w'||_{\Q}=1$, whence $v',w'\in\Zn$.

\medskip

\noindent{\sf Claim 2:} $\conv(0,v',v_1,\ldots,v_{n-1})$
 and $\conv(0,w',w_1,\ldots,w_{n-1})$ are regular.
 
 As a matter of fact, again by  Theorem~\ref{theorem:volumi}(vi)  we can write

\begin{align*}
\det\left(\begin{tabular}{cl}
$0\ \ldots\ 0$&1\\
$v'$ & $1$\\
$d_F v_1$& $d_F$\\
\vdots&\vdots\\
$d_F v_{n-1}$&$d_F$
\end{tabular}
\right)&=
\det\left(\begin{tabular}{c}
$v'$ \\
$d_F v_1$\\
\vdots\\
$d_F v_{n-1}$
\end{tabular}
\right)
=d_F^{n-1}
\det\left(\begin{tabular}{c}
$v'$ \\
$v_1$\\
\vdots\\
$v_{n-1}$
\end{tabular}
\right)
\\
&=\frac{d_F^{n-1}}{||v||_{\Q}}
\det\left(\begin{tabular}{c}
$v$ \\
$v_1$\\
\vdots\\
$v_{n-1}$
\end{tabular}
\right)
=\frac{d_F^{n-1}\mathsf V_F}{||v||_{\Q}}  =1
\end{align*}
The same argument proves that $\conv(0,w',w_1,\ldots,w_{n-1})$ is regular, which settles our claim.
\medskip

Another application of  Lemma~\ref{Lemma:GammaFromSimplex}  yields $\gamma\in\GLn$ such that $\gamma(v')=w'$ and $\gamma(v_i)=w_i$ 
for each $i=1,\ldots,n-1$. 
Then \[\gamma(v)=\gamma(||v||_{\Q}v')=||v||_{\Q}\gamma(v')=||w||_{\Q} w'=w,\]
 and 
\[
\gamma(F)= \gamma(\aff(v,v_1,\ldots,v_{n-1}))
=\aff(\gamma(v),\gamma(v_1),\ldots, \gamma(v_{n-1}))=G.\qedhere
\]

\end{proof}

\end{document}